\documentclass[a4paper,12pt]{article}
\usepackage{times, url}
\usepackage{amsmath,amssymb,amsthm,amsfonts,bm,float,cite}

\newtheorem{theorem}{Theorem}[section]
\newtheorem{corollary}{Corollary}[section]
\newtheorem{lemma}{Lemma}[section]

\newtheorem{proposition}{Proposition}
\theoremstyle{definition}

\usepackage[none]{hyphenat}

\usepackage[center]{caption}
\usepackage{color, colortbl}
\definecolor{Gray}{gray}{0.9}

\usepackage[utf8]{inputenc}
\usepackage{amsfonts}
\usepackage{amsmath}
\usepackage{amsthm}
\usepackage{blindtext}
\usepackage{graphicx}
\usepackage[numbers]{natbib}
\usepackage{amssymb}

\newcommand{\ZZ}{\mathbb{Z}}

\newcommand{\mmod}{\,\mathrm{mod}\,}
\newcommand{\ssum}{\sideset{}{^*}\sum}

\begin{document}
\setcounter{page}{1}

\begin{center}
{\LARGE \bf  Partial Liouville sums over divisors \\[2mm]} 
\vspace{8mm}

{\Large \bf Milo Moses}
\vspace{3mm}

Berkeley High School \\ 
e-mail: \url{milo@tacocat.com}
\vspace{2mm}

\end{center}
\vspace{6mm}

\noindent
{\bf Abstract:} Using sieves and elementary manipulations, we show that the signs of partial sums of the Liouville function over divisors are in a strong sense equally distributed.
\vspace{2mm}

\section{Introduction} \label{sec:Intr}

The Liouville function $\lambda(n)$ is defined by $\lambda(n)=(-1)^{\Omega(n)}$ where $\Omega(n)$ counts the total number of factors in the prime decomposition of $n$. A natural problem is to study the partial sums $L(n,z):=\sum_{d|n,\,\, d< z}\lambda(n)$. These sums will be larger if divisors come clumped in groups with the same parity of number of prime divisors, and they will be smaller otherwise. In \cite{drerss1983somme}, it is proved that the quantities

$$\lim_{x\to\infty}x^{-1}\sum_{n< x}L(n,z)^2$$

exist for each $z$, and converge to a finite limit as $z$ tends to infinity. These quantities were further studied in \cite{breteche2020remarques}. The purpose of this note is to show that the signs of $L(n,z)$ are randomly distributed over $n$ and $z$, in the following sense:

\begin{theorem}\label{THEOREM} Let $(a_{n,x})_{n\geq 1}$ be a sequence of complex numbers depending on $x$ such that

$$\limsup_{x\to\infty} x^{-1}\sum_{n<x}|a_{n,x}|^2<\infty.$$

If the limits

$$\lim_{x\to\infty}x^{-1}\sum_{n< x} a_{n,x} L(n,z)$$

exist for all $z$, then they tend to $0$ as $z$ tends to infinity.
\end{theorem}

\section{Details and proof}

We begin by clarifying the condition on $a_{n,x}$ used in Theorem \ref{THEOREM}.

\begin{proposition}\label{condition} The quantities $\lim_{x\to\infty}x^{-1}\sum_{n< x} a_{n,x} L(n,z)$ exist for all $z$ if and only if the quantities

$$\lim_{x\to\infty} x^{-1}\sum_{n< x} a_{qn,x}$$

exist for all integers $q\geq 1$.
\end{proposition}
\begin{proof} Manipulating, we find

\begin{align*}
\lim_{x\to\infty}x^{-1}\sum_{n< x} a_n L(n,z)&=\lim_{x\to\infty}x^{-1}\sum_{n< x} a_n \sum_{d|n, \,\, d< z} \lambda(d)\\
&=\lim_{x\to\infty} x^{-1}\sum_{d< z}\sum_{n< x/d} a_{dn}\lambda(d).
\end{align*}

Thus, letting $z$ vary, we find that

$$\lim_{x\to\infty} x^{-1}\sum_{n< x/d} a_{dn}\lambda(d)$$

must exist for all $d$. Since $\lambda(d)$ is nonzero, we arrive at the desired conclusion.
\end{proof}

We now prove a technical lemma:

\begin{lemma}\label{LEMMA} Fix a sequence $(a_{n,x})_{n\geq 0}$. Define

$$f_x(q)=\frac{q}{x}\sum_{n< x/q}a_{qn,x},\,\,\, g_x(q)=\sum_{d|q}\mu(q/d)f_x(d),$$

and $S_x(b/q)=x^{-1}\sum_{n<x}a_{n,x} e^{2\pi n (b/q)i}$. The following equality holds:

$$g_x(q)=x^{-1}\ssum_{b\mmod q}S_x(b/q).$$

The star indicates the summation is taken over residue classes in $(\ZZ/q\ZZ)^*$.
\end{lemma}
\begin{proof} We obverse via sum manipulations that

\begin{align*}
x^{-1}\sum_{d|q}\ssum_{b\mmod d}S_x(b/d)&=x^{-1}\sum_{b\mmod q}S_x(b/q)\\
&=x^{-1}\sum_{n<x}a_{n}\left(\sum_{b\mmod q} e^{2\pi  n (b/q)i}\right)\\
&=\frac{q}{x}\sum_{n<x/q}a_{qn}
\end{align*}

The desired formula follows by M\"{o}bius inversion.
\end{proof}

\begin{corollary}\label{COROLLARY} Using the notation of Lemma \ref{LEMMA}, if $(a_{n,x})$ is such that $g(q)=\lim_{x\to\infty}g_x(q)$ exists for all $q$, then

$$\sum_{n=1}^{\infty}\frac{|g(q)|^2}{\varphi(q)}$$

converges whenever $\limsup_{x\to\infty}x^{-1}\sum_{n<x}|a_{n,x}|^2$ converges
\end{corollary}
\begin{proof} Using Cauchy–Schwarz on $g(q)$, we find

$$\sum_{q< Q}\frac{|g(q)|^2}{\varphi(q)}\leq \limsup_{x\to\infty} x^{-2}\sum_{q< Q} \,\ssum_{b\mmod q}|S_x(b/q)|^2$$

The large sieve inequality states that the left hand side of this expression is bounded above by

$$\limsup_{x\to\infty} \frac{x+Q}{x^2}\sum_{n< x}|a_{n,x}|^2,$$

which is bounded uniformly as $Q$ varies by our assumption on $(a_{n,x})$.
\end{proof}

We can now prove the main theorem:

\begin{proof}[Proof of Theorem \ref{THEOREM}] We work with the notation of Lemma \ref{LEMMA}. To begin, we see via elementary sum manipulations

\begin{align*}
\lim_{x\to\infty}x^{-1}\sum_{n< x}a_{n,x}L(n,z)&=\lim_{x\to\infty}x^{-1}\sum_{n< x}a_{n,x}\sum_{d|n,\,\, d< z}\lambda(d)\\
&=\sum_{d< z}\frac{\lambda(d)f(d)}{d},
\end{align*}

where $f(d)=\lim_{x\to\infty}f_x(d)$. This limit exists for all $d$ by Proposition \ref{condition}. By M\"{o}bius inversion, we know that $g(d)=\lim_{x\to\infty}g_x(d)$ must exist as well. Thus, we can manipulate our sum further as

\begin{align*}
\sum_{d< z}\frac{\lambda(d)f(d)}{d}&=\sum_{d< z}\frac{\lambda(d)}{d}\left(\sum_{q|d}g(q)\right)\\
&=\sum_{q< z}\frac{\lambda(q)g(q)}{q}\left(\sum_{d< z/q}\frac{\lambda(d)}{d}\right).
\end{align*}

Note the key use of the fact that $\lambda$ is completely multiplicative. Combining our work thus far, we get the following:

\begin{equation*}
\lim_{x\to\infty}x^{-1}\sum_{n< x}a_nL(n,z)=\sum_{q< z}\frac{\lambda(q)g(q)}{q}\left(\sum_{d< z/q}\frac{\lambda(d)}{d}\right)
\end{equation*}

By the prime number theorem $\left| \sum_{d< z/q}\frac{\lambda(d)}{d} \right| \ll \frac{1}{\log^*(z/q)}$ where $\log^*(z/q)=\max(1,\log(z/q))$. Hence,

\begin{equation*}
\left|\sum_{q< z}\frac{\lambda(q)g(q)}{q}\left(\sum_{d< z/q}\frac{\lambda(d)}{d}\right)\right|\ll \sum_{q< z} \frac{|g(q)|}{q\log^*(z/q)}.
\end{equation*}

Fixing large $T>0$, we find by Cauchy–Schwarz that

\begin{align*}
\left|\sum_{q< z} \frac{|g(q)|}{q\log^*(z/q)}\right|&\leq \left|\sum_{q< z/T} \frac{|g(q)|}{q\log^*(z/q)}\right|+\left|\sum_{z/T\leq q < z} \frac{|g(q)|}{q\log^*(z/q)}\right|\\
&\leq \left|\sum_{q< z/T} \frac{|g(q)|^2}{\varphi(q)}\right|^{1/2}\cdot \left|\sum_{q< z/T} \frac{\varphi(q)}{q^2\log^*(z/q)^2}\right|^{1/2}\\
&+\left|\sum_{z/T\leq z< z} \frac{|g(q)|^2}{\varphi(q)}\right|^{1/2}\cdot \left|\sum_{z/T\leq q< z} \frac{\varphi(q)}{q^2\log^*(z/q)^2}\right|^{1/2}\\
\end{align*}

We treat these summations one by one. By the conditions of the proposition we get that the first sum is bounded uniformly in terms of $a_n$, and that the first sum on the second row is $o_T(1)$. For the second sum in the first row, we note that $\log^*(z/g)\geq \log(T)$. The second sum in the bottom row is clearly $O_T(1)$, and hence collecting we get that

\begin{equation*}
\limsup_{z\to\infty}\left|\sum_{q< z} \frac{|g(q)|}{q\log^*(z/q)}\right|\ll \limsup_{z\to\infty}\left|\sum_{q< z/T} \frac{1}{q\log^*(z/q)^2}\right|^{1/2}.
\end{equation*}

Decomposing along intervals $2^{-(k+1)}\cdot z/T < q \leq 2^{-k}\cdot z/T$ it is clear that not only is the right hand side bounded but it tends to $0$ as $T\to\infty$. Hence, we conclude the result.
\end{proof}

\section{Acknowledgements} 

The author thanks T. Tao for his collaboration on part of the proof of Theorem \ref{THEOREM}, and thanks K. Soundararajan for making him aware of the reference \cite{drerss1983somme} which contains results proved in an earlier draft.

\bibliographystyle{plain}
\bibliography{ref}

\end{document}